\documentclass[12pt]{article} 

\usepackage{amsmath}
\usepackage{amssymb}
\usepackage{amsthm}
\usepackage[all,cmtip]{xy}
\usepackage{fancyhdr}
\usepackage{euscript}
\usepackage{setspace}
\usepackage{graphicx,epsfig}
\usepackage{palatino, url, multicol}
\usepackage{mathrsfs} 
\usepackage{enumerate}   
\usepackage{verbatim} 

\theoremstyle{definition}

\newtheorem{theorem}{Theorem}[section]

\newcommand{\IN}{\hbox{$\mathbb{N}$}}
\newcommand{\IZ}{\hbox{$\mathbb{Z}$}}

\newcommand{\IQ}{\hbox{$\mathbb{Q}$}}

\renewcommand{\hat}{\widehat}

\def\it{\itshape}

\def\tt{\texttt}
\def\bf{\textbf}

\def\IZ{{\mathbb{Z}}}

\def\IQ{{\mathbb{Q}}}
\def\IN{{\mathbb{N}}}

\def\I1{{\mathbb{1}}}

\def\limx0{\lim_{x \to 0}}

\def\intxyleq1{\underset{\| x - y  \| \leq 1}{\int}}
\def\intxygeq1{\underset{\| x - y  \| \geq 1}{\int}}
\def\intxizetaleq1{\underset{\| \xi - \zeta  \| \leq 1}{\int}}
\def\intxizetageq1{\underset{\| \xi - \zeta \| \geq 1}{\int}}

\def\tab{\hskip 1mm}

\def\tab{\hspace{.1pc}}
\def\ttab{\hspace{1pc}}

\newcounter{hours}
\newcounter{minutes}
\newcommand\printtime{%
  \setcounter{hours}{\the\time/60}%
  \setcounter{minutes}{\the\time-\value{hours}*60}%
  \ifthenelse{\value{hours} > 12}
     {
       \setcounter{hours}{\value{hours}-12}%
       \thehours:\theminutes \ p.m.                
     }
     {
       \thehours:\theminutes \ a.m.                
     } 
}

\def\putdate{{\tt Compiled on \the\month-\the\day-\the\year \ at\printtime} \\}

\usepackage[linesnumbered,ruled]{algorithm2e}

\begin{document} 

\title{Irrational Base Counting}
 \author{Avraham Bourla\\
   Department of Mathematics\\
  Brandeis University, Waltham MA\\
   \texttt{abourla@brandeis.edu}}
 \date{\today}
 \maketitle
\begin{abstract}
\noindent We will provide algorithmic implementation with proofs of existence and uniqueness for the absolute and alternating irrational base numeration systems.
\end{abstract} 

\section{Introduction}{}

We can view a positive integer written in our familiar base--10 numeration system as the dot product of a finite sequence of digits $(d_k)_1^\ell \subset \{0,1,...,9\}$ and the infinite base--$10$ vector $(10^k)_0^\infty$ truncated to the $\ell-1$ position. For instance when $\ell=3$ and $(d_k:=k)_1^3$, we have
\[\sum_{k=1}^\ell{d_k}10^{k-1} = (1,2,3)\cdot(10^k)_0^2 = 1\cdot{10}^0 + 2\cdot10^1+3\cdot10^2 =321.\]
After taking zero as the vacuous expansion obtained when $\ell=0$ and allowing the infinite base--$10$ vector to alternate in sign as $((-10)^k)_0^\infty$, we can expand all integers base--$(-10)$. For instance, $-321 = (9,3,7,1)\cdot((-10)^k)_0^3$, whereas $321$ is now given the new digit representation $(1,8,4)$. We can similarly obtain integer expansions for all fix radix base--n systems. In this paper, we how show how to expand integers as a dot product using an irrational base. The idea behind these expansions date back to Ostrowski \cite{Ost}, who used the continued fraction expansion as a tool in inhomogeneous Diophantine Approximation.\\

\noindent After fixing the base $\alpha \in (0,1)\backslash\IQ$, we expand it as an infinite continued fraction
\[\alpha = \dfrac{1}{a_1 + \dfrac{1}{a_2 + \dfrac{1}{a_3+\dfrac{1}{\ttab\ddots}}}},\]
obtaining the unique sequence of partial quotients $(a_k)_1^\infty$ (for details refer to any of the standard introductions \cite{HW,K}). Truncating the iteration after $k$ steps yields the convergent
\[\frac{p_k}{q_k} := \dfrac{1}{a_1 + \dfrac{1}{a_2 +\dfrac{1}{\ddots \begin{matrix}\\ +\dfrac{1}{a_k} \end{matrix}}}}.\]
We will utilize the sequence of denominators $(q_k)_0^\infty$ as the infinite base--$(\alpha)$ vector and the alternating sequence $((-1)^k{q_k})_0^\infty$ as the base--$(-\alpha)$ vector, providing rigorous proofs of existence as well as concrete algorithmic realization and some counting examples. We end this section by quoting the well known recursion equation
\begin{equation}\label{q_k_recursion}
q_{-1} := 0, \ttab q_0 :=1,\ttab q_k =  a_k{q_{k-1}} + q_{k-2} \ttab k \ge 1.
\end{equation}
After we define
\begin{equation}\label{q_k^*}
q_k^* := (-1)^k{q_k}, \ttab k \ge -1,
\end{equation}
we use this relationship to obtain the new recursion equation 
\begin{equation}\label{q_k^*_recursion}
q_{-1}^* := 0, \ttab q_0 := 1, \ttab q_k^* = q_{k-2}^*-a_k{q_{k-1}}^*, \ttab k \ge 1.
\end{equation}

\section{The Base--$\alpha$ Expansion}

\subsection{Algorithm and proof}
The base--$\alpha$ expansion is of the dot product of the sequence of digits $(c_k)_1^\ell$, where $\ell \in \IN$ and the infinite sequence $(q_k)_0^\infty$ truncated to the $\ell-1$ position. We say that the digit sequence $(c_k)_1^\infty\subset\IN$ is $\boldsymbol{\alpha}$\bf{--admissible} when it satisfies the following Markov conditions:
\begin{itemize}
\item $c_1 \le a_1 -1$ and $c_k \le a_k$ for $k \ge 1$, not all zeros.
\item If $c_k = a_k$ then $c_{k-1}=0$.
\end{itemize}
\begin{theorem}\label{base_N}
For every $N \in \IN$ there exists $\ell\ge0$ and a unique $\alpha$--admissible sequence of digits $(c_k)_1^\ell$ such that $N = \sum_{k=1}^{\ell}{c_k}q_{k-1}$. 
\end{theorem}

\begin{proof}

\noindent Apply the algorithm:\\

\IncMargin{1em}
\begin{algorithm}[H]
\SetKwInOut{Input}{input}\SetKwInOut{Output}{output}
\Input{$\alpha\in(0,1)\backslash\IQ,\tab N\in\IN_{\ge 0}$}
\Output{$\ell \in \IN, (c_k)_1^\ell \text{ }\alpha$--admissible}
set $N_0 := N, m=n_0:= 0$\;
\While{$N_m \ge 1$}
{
let $n_m$ be such that $q_{n_m-1} \le N_m < q_{n_m}$\;
set $c_{n_m} := \lfloor N_m/{q_{n_m-1}}\rfloor$\;
set $N_{m+1} := N_m - c_{n_m}q_{n_m-1}$\;
set $m:=m+1$\;
}
set $M := m, \ell := n_0, c_k := 0$ for all $k \notin \{n_m\}_0^M$\;
\caption{Natural Expansion}\label{abs}
\end{algorithm}\DecMargin{1em}
\vspace{1pc}

\noindent When $N=0$, we have $\ell=n_0=0$ and the expansion is vacuous. Whenever $N_m \ge 1$, we see that since $q_0=1$ by definition \eqref{q_k_recursion}, the assignment of step--3 and the step--4 guarantees that $n_m \ge 1$ and that
\begin{equation}\label{c_n_m>0}
c_{n_m}\ge 1.
\end{equation}
After we rewrite the assignment of line--4 as the inequality 
\begin{equation}\label{c_n_m}
c_{n_m}q_{n_m-1} \le N_m < (c_{n_m}+1)q_{n_m-1},
\end{equation}
we observe that, in tandem with the assignment of line-5, we are applying the euclidean algorithm as the repeated integer division of $N_m$ by $q_{n_m-1}$ resulting in a quotient $c_{n_m}$ and remainder $N_{m+1}$. Thus we must have $0 \le N_{m+1} < N_m \le N$, that is, this iteration scheme must eventually terminate with a finite positive value $M$, yielding the sequences 
\[0= N_M < N_{M-1} < ... < N_0 = N, \ttab 0 \le n_M < ... < n_1 < n_0 = \ell \ttab \text{ and } \ttab (c_{n_m})_{m=0}^{M-1}.\] 
For all $1 \le k \le \ell$ with $k \notin \{n_m\}_0^{M-1}$ we define $c_k : =0$ and then, using the assignment of step--6, we obtain the desired expansion
\[N = N_0 = c_{n_0}q_{n_0-1} + N_1 =  c_{n_0}q_{n_0-1} + c_{n_1}q_{n_1-1} + N_2 = ... = \sum_{m=0}^{M-1}c_{n_m}q_{n_m-1} = \sum_{k=1}^{\ell}c_k{q_{k-1}}.\]
Furthermore, the uniqueness of the quotient and the remainder terms in the division algorithm guarantees the uniqueness of this expansion.\\

\noindent If $M$ is such that $n_M \ge 2$ then $c_1=0$ and if $n_M=1$, we use the fact that $q_0=1$ and the inequality \eqref{c_n_m} to verify that $c_1 = c_1{q_0} \le N_1 < q_1 = a_1$. Conclude that $c_1 \le a_1-1$ as desired. If for some $m$ we have in step 2 that $c_{n_m} \ge a_{n_m}+1$, then the recursion formula \eqref{q_k_recursion}, the inequality \eqref{c_n_m} and the fact that the sequence $(q_k)_0^\infty$ is strictly increasing will lead us to the contradiction
\[N_m < q_{n_m} = a_{n_m}q_{n_m-1}+q_{n_m-2} < (a_{n_m} + 1)q_{n_m-1} \le c_{n_m}q_{n_m-1} \le N_m.\] 
Therefore, for all $k$ we must have $0 \le c_k \le a_k$. Next, suppose by contradiction that $c_k = a_k$ and $c_{k-1} \ge 1$. Since $c_k=a_k \ge 1$, we see from the inequality \eqref{c_n_m>0} that there is some $m$ for which $n_m=k-1$. The the recursion formula \eqref{q_k_recursion}, the inequality \eqref{c_n_m} and the assignment of line--5 will now leads us to the contradiction
\[N_m < q_{n_m} = q_{k-1} < q_k =  q_k - N_{m+1} + N_{m+1} \le q_k - c_{n_m+1}q_{n_m} + N_{m+1} \le q_k - c_k{q_{k-1}} + N_{m+1}\]
\[=q_k - a_k{q_{k-1}}+ N_{m+1} = q_{k-2}+ N_{m+1} \le c_{k-1}q_{k-2} + N_{m+1} = c_{n_m}q_{n_m-1} + N_{m+1} = N_m.\] 
\end{proof}

\subsection{Examples}

When 
\[\alpha := .5(5^{.5}-1) = \dfrac{1}{1+\dfrac{1}{1+\dfrac{1}{\ddots}}}\]
is the golden section, we have $\{a_k\}_1^\infty = \{1\}$. We then use formula \eqref{q_k_recursion} to verify that the sequence $(q_k)_0^\infty$ is no other than the Fibonacci Sequence $(F_k)_0^\infty := (1,1,2,3,5,8,13,...)$. The implication of the proposition to this case is the Zeckendorf Theorem, which states that every positive integer can be uniquely written as the sum of nonconsecutive terms in $(F_k)_1^\infty$.

\noindent When 
\[\alpha := \sqrt{2}-1 = \dfrac{1}{2+\dfrac{1}{2+\dfrac{1}{\ddots}}}\]
is the sliver section, we have $\{a_k\}_1^\infty=\{2\}$. By formula \eqref{q_k_recursion}, we verify that $(q_k)_0^3 = (1,2,5,12)$. The following tables display how the digits behave when we count to twenty four using this base:\\

\small
\begin{tabular}{|c||c|c|c|c|}
\hline

& $q_3=12$ & $q_2 = 5$ & $q_1 = 2$ & $q_0=1$\\
$N$&  $c_4$ & $c_3$ & $c_2$ & $c_1$\\
\hline\hline
1 & 0 & 0 & 0 & 1 \\
\hline
2 & 0 &0 & 1 & 0\\
\hline
3 &0 & 0 & 1 & 1\\
\hline
4 &0 & 0 & 2 & 0\\
\hline
5 & 0 &1 & 0 & 0\\
\hline
6 &0 & 1 & 0 & 1 \\
\hline
7 &0 & 1 & 1 & 0\\
\hline
8 &0 & 1 & 1 & 1\\
\hline
9 &0 & 1 & 2 & 0\\
\hline
10 &0 & 2 & 0 & 0\\
\hline
11 &  0 & 2 & 0 & 1\\
\hline
12 & 1 & 0 & 0 & 0\\
\hline
\end{tabular}\hspace{2pc}\begin{tabular}{|c||c|c|c|c|c|}
\hline
& $q_3=12$ & $q_2 = 5$ & $q_1 = 2$ & $q_0=1$\\
$N$ & $c_4$ & $c_3$ & $c_2$ & $c_1$\\
\hline\hline
13 & 1 & 0 & 0 & 1\\
\hline
14 & 1 & 0 & 1 & 0\\
\hline
15 & 1 & 0 & 1 & 1\\
\hline
16 & 1 & 0 & 2 & 0\\
\hline
17 & 1 & 1 & 0 & 0\\
\hline
18 & 1 & 1 & 0 & 1\\
\hline
19 & 1 & 1 & 1 & 0\\
\hline
20 & 1 & 1 & 1 & 1\\
\hline
21 & 1 & 1 & 2 & 0\\
\hline
22 & 1 & 2 & 0 & 0\\
\hline
23 & 1 & 2 & 0 & 1\\
\hline
24 & 2 & 0 & 0 & 0\\
\hline
\end{tabular}\hspace{2pc}
\vspace{1pc}

\normalsize

\section{The Base--(-$\alpha)$ Expansion}
\subsection{Algorithm and proof}
The base--$(-\alpha)$ expansion is of the dot product of the sequence of digits $(b_k)_1^\ell$, where $\ell \in \IN$ and the infinite sequence $(q_k^*)_0^\infty$ truncated to the $\ell-1$ position. We say that the digit sequence $(b_k)_1^\infty\subset \IN$ is $\boldsymbol{(-\alpha)}$\bf{--admissible} when:
\begin{itemize}
\item $b_k \le a_k$ not all zeros.
\item If $b_k = a_k$ then $b_{k+1}=0$.
\end{itemize}
\begin{theorem}\label{base_Z}
For every integer $Z$ there is some $\ell\ge0$ and a unique $(-\alpha)$--admissible sequence of digits $(b_k)_1^\ell$ such that $Z = \sum_{k=1}^{\ell}{b_k}q_{k-1}^*$. 
\end{theorem}

\begin{proof}

\noindent We let $I_{R}$ be the indicator function for the relationship $R$ and apply the algorithm:\\

\IncMargin{1em}
\begin{algorithm}[H]
\SetKwInOut{Input}{input}\SetKwInOut{Output}{output}
\Input{$Z \in\IZ$, $\alpha\in(0,1)\backslash\IQ$}
\Output{$\ell \in \IN, (b_k)_1^\ell (-\alpha)$--admissible}
set $Z_0 := Z, m=b_1=n_0 := 0$\;
\While{$Z_m \ne 0$}
{
let $n_m' \ge 0$ be such that $q_{n_m'-1} < |Z_m| + I_{<0}(Z_m) \le q_{n_m'}$\; 
let $n_m \in \{n_m', n_m'+1\}$ be such that $I_{>0}\left((-1)^{n_m-1}Z_m\right) = 1$\; 
\uIf{$n_m=n_m'$}
{
set $b_{n_m}' := \lfloor |Z_m|/q_{n_m-1}\rfloor$\;
\uIf{$|Z_m-b_{n_m}'q_{n_m-1}^*| + I_{<0}(Z_m-b_{n_m}'q_{n_m-1}^*)  > q_{n_m-2}$}
{
set $b_{n_m} := b_{n_m}'+1$\;
}
\Else
{
set $b_{n_m} := b_{n_m}'$\;
}}
\Else
{
set $b_{n_m} := 1$\;
}
set $Z_{m+1} := Z_m - b_{n_m}q_{n_m-1}^*$\; 
set $m := m+1$\;
}
set $M := m, \ell := n_0, b_1 := b_1 + Z_m, b_k := 0$ for all $k \notin \{n_m\}_0^M$\;
\caption{Integer Expansion}\label{alt}
\end{algorithm}\DecMargin{1em}
\vspace{1pc}

\noindent The definition \eqref{q_k^*} of $q_k^*$ and the assignment of line--4 provides us with the inequality
\begin{equation}\label{Z_m_q_n_m-1}
Z_m{q_{n_m-1}}^* = (-1)^{n_m-1}Z_m{q_{n_m-1}} \ge 0,
\end{equation}
whereas the assignment of line--6 provides us with the inequality
\begin{equation}\label{b_n_m}
b_{n_m}'q_{n_m-1} \le |Z_m| < (b_{n_m}'+1)q_{n_m-1}. 
\end{equation}
When $Z_0=0$, we have $\ell=0$ and  the expansion is vacuous. Assuming $Z_0 \ne 0$, we will first show that the sequence of indexes $(n_m)_0^M$ is strictly decreasing. To do so, we will consider the two cases $n_m' \in \{n_m-1,n_m\}$ separately:\\

\noindent $\bullet$ When $n_m'=n_m-1$, the inequality of step 3 yields 
\[q_{n_m-2} = q_{n_m'-1} < |Z_m| \le |Z_m| + I_{<0}(Z_m) \le q_{n_m'} = q_{n_m-1},\]
so when we define $Z_{m+1}$ using $b_{n_m}=1$ in step 15, we will have by the inequalities \eqref{Z_m_q_n_m-1} and \eqref{b_n_m} that 
\begin{equation}\label{opposite_signs}
|Z_m - b_{n_m}'{q_{n_m-1}}^*| = |Z_m| - b_{n_m}'{q_{n_m-1}}
\end{equation}
and that $Z_m{Z_{m+1}} \le 0$, hence
\begin{equation}\label{Z_m+1_case1}
|Z_{m+1}| = q_{n_m-1} - |Z_m| \le q_{n_m-1} - q_{n_m-2} - 1.
\end{equation}
Since $n_m = n_m'+1 \ge 1$, we have $q_{n_m-2} \ge 1$, so that $|Z_{m+1}| + 1 \le q_{n_m-1}$ and
\[|Z_{m+1}| + I_{<0}(Z_{m+1}) \le |Z_{m+1}| + 1 \le q_{n_m-1}.\]
Then in step 3 of the next iteration, we will have $n_{m+1}' \le n_m-1$. If this inequality is strict then we have $n_{m+1} \le n_{m+1}'+1 < n_m$. If $n_{m+1}' = n_m+1$, then in step 4 we use the fact that $Z_m$ and $Z_{m+1}$ are of opposite sign to obtain  
\[I_{>0}\left((-1)^{n_m}Z_{m+1}\right) = I_{>0}\left((-1)^{n_m-1}Z_m\right) = 1 = I_{>0}\left((-1)^{n_{m+1}-1}Z_{m+1}\right),\]
that is, 
\[n_{m+1}'-1 \equiv n_m \equiv n_{m+1}-1 \ttab (\operatorname{mod} 2).\] 
Since $n_{m+1} \le n_{m+1}'+1 \le n_m$, we conclude that for this case we have $n_{m+1} = n_{m+1}' < n_m$.\\

\noindent $\bullet$ When $n_m' = n_m$ and $Z_m > 0$, we have by the inequalities \eqref{Z_m_q_n_m-1}, \eqref{b_n_m}, line--15 and the fact that $0 \le b_{n_m} - b_{n_m}'\le1$ that
\[Z_{m+1} = Z_m - b_{n_m}q_{n_m-1}^* = |Z_m| - b_{n_m}q_{n_m-1} < (b_{n_m}'+1)q_{n_m-1}-b_{n_m}'q_{n_m-1} = q_{n_m-1}\]
and
\[-q_{n_m-1} = b_{n_m}'q_{n_m-1} - (b_{n_m}'+1)q_{n_m-1} \le b_{n_m}'q_{n_m-1} - b_{n_m}q_{n_m-1}\] 
\[\le |Z_m| - b_{n_m}q_{n_m-1} = Z_m - b_{n_m}q_{n_m-1}^* = Z_{m+1}.\]
Similarly, when $n_m' = n_m$ and $Z_m < 0$, we have by the inequalities \eqref{Z_m_q_n_m-1}, \eqref{b_n_m}, line--15 and the fact that $0 \le b_{n_m} - b_{n_m}'\le1$ that
\[Z_{m+1} = Z_m - b_{n_m}q_{n_m-1}^* = -|Z_m| + b_{n_m}q_{n_m-1} \le - b_{n_m}'q_{n_m-1} + (b_{n_m}'+1)q_{n_m-1} = q_{n_m-1}\]
and
\[-q_{n_m-1} = -(b_{n_m}'+1)q_{n_m-1} + b_{n_m}'q_{n_m-1} < -|Z_m| + b_{n_m}'q_{n_m-1} \le Z_m - b_{n_m}q_{n_m-1}^* = Z_{m+1}\]
In either case we have 
\begin{equation}\label{Z_m+1_case2}
|Z_{m+1}| \le q_{n_m-1}.
\end{equation}
If one of the last inequalities is an equality, then the iteration will terminate at the next step with $n_{m+1} = n_m, \tab b_{n_{m+1}}=1$ and $Z_{m+2} = 0$. Otherwise, we have $|Z_{m+1}| + I_{<0}(Z_{m+1}) \le q_{n_m-1}$ so that by line--3 we will have $n_{m+1}' \le n_m-1$. When $n_{m+1} = n_{m+1}'$, we have $n_{m+1} < n_m$ and when $n_{m+1} - 1= n_{m+1}'$ we use the previous paragraph to conclude that $n_{m+2} < n_{m+1}$. In either case we have $n_{m+2} \le  n_{m+1} \le n_m$ and $n_{m+2} < n_m$.\\ 

\noindent We have just proved that the sequence $(n_m)_0^M$ is non-constant and decreasing and thus conclude that this iteration process will eventually terminate with a finite value $M$, for which $n_M \ge 1$ and $Z_{M+1} =0$. After we define $b_k := 0$ whenever $k \notin \{n_m\}_0^M$, we use the assignment of line--15 to obtain the desired expansion 
\[Z_0 =  b_{n_0}q_{n_0-1}^* + Z_1 = b_{n_0}q_{n_0-1}^* + b_{n_1}q_{n_1-1}^* + Z_2 = ... = \sum_{k=1}^{\ell}{b_k}q_{k-1}^*.\]
To prove uniqueness, we split an expansion of $Z_0$ into its positive and negative parts and invoke the uniqueness of the absolute irrational expansion. More precisely, if $Z_0 = \sum_{k=1}^{\ell}{b_k}q_{k-1}^* $, then we define
\[Z_0^+ := \sum_{k=0}^{\lceil\ell/2\rceil}{b_{2k+1}}q_{2k}^* = \sum_{k=0}^{\lceil\ell/2\rceil}{b_{2k+1}}q_{2k}, \ttab Z_0^- := -\sum_{k=1}^{\lceil\ell/2\rceil}{b_{2k}}q_{2k-1}^* = \sum_{k=1}^{\lceil\ell/2\rceil}{b_{2k}}q_{2k-1},\]
so that $Z_0 = Z_0^+ - Z_0^-$. If we also have $Z_0 = \sum_{k=1}^{\hat{\ell}}{\hat{b}_k}q_{k-1}^*$ then, without changing the representation, we set $b_k=\hat{b}_k:=0$ for all $\min\{\ell,\hat{\ell}\} < k \le \max\{\ell,\hat{\ell}\}$ and write
\[\sum_{k=1}^{\lceil\ell/2\rceil}{b_{2k}}q_{2k-1} = Z_0^- = Z_0^+ - Z_0 = \sum_{k=0}^{\lceil{\ell}/2\rceil}{b_{2k+1}}q_{2k}-\sum_{k=1}^{\hat{\ell}}{\hat{b_k}}q_{k-1}^*\]  
\[= \sum_{k=0}^{\lceil{\max\{\ell,\hat{\ell}\}}/2\rceil}{(b_{2k+1}-\hat{b}_{2k+1})}q_{2k}   + \sum_{k=1}^{\lceil{\hat{\ell}}/2\rceil}{\hat{b}_{2k}}q_{2k-1}.\]
Then theorem \ref{base_N} guarantees that $\ell = \hat{\ell}$ and that $b_k = \hat{b}_k$ for all $1 \le k \le \ell$.\\

\noindent To prove that for all $k \ge 1$ we have $b_k \le a_k$, we will show that for all $0\le m \le M$ we have $0 \le b_{n_m} \le a_{n_m}$. This is clear whenever $n_m=n_m'+1$ for by the assignment of line--13, we have $b_{n_m} = 1$. When $n_m=n_m'$, we use the inequality of line--3 and the assignments of line--6, line--8 and line--10, we see that $b_{n_m} \ge b_{n_m}' \ge 1$. Furthermore, we cannot have $b_{n_m}' \ge a_{n_m} + 1$, for then we would use the recursion relationship \eqref{q_k_recursion} and the inequalities of line--3 and \eqref{b_n_m} to obtain the contradiction
\[|Z_m| \le q_{n_m} - I_{<0}(Z_m) \le q_{n_m} = a_{n_m}q_{n_m-1} + q_{n_m-2} \le (b_{n_m}'-1)q_{n_m-1} + q_{n_m-2}\] 
\[= b_{n_m}'q_{n_m-1} - (q_{n_m-1} - q_{n_m-2}) < b_{n_m}'q_{n_m-1} \le |Z_m|. \]
Finally, when $b_{n_m}' = a_{n_m}$, we will show that we must also have $b_{n_m} = a_{n_m}$. If $Z_m > 0$, then from line--4 and the definition \eqref{q_k^*} of $q_k^*$ we have $(-1)^{n_m-1} = 1$ and $q_{n_m-1}^* = q_{n_m-1}$ so that by the inequality \eqref{b_n_m} we obtain 
\[Z_m-b_{n_m}'q_{n_m-1}^* = |Z_m| - b_{n_m}'q_{n_m-1} \ge 0.\] 
Then the the recursion relationship \eqref{q_k_recursion} and the inequality of line--3 will now yield the inequality
\[|Z_m - b_{n_m}'q_{n_m-1}^*| + I_{<0}(Z_m- b_{n_m}'q_{n_m-1}^*) = |Z_m - b_{n_m}'q_{n_m-1}^*|\] 
\[= Z_m - b_{n_m}q_{n_m-1} = Z_m - a_{n_m}q_{n_m-1} \le q_{n_m} - a_{n_m}q_{n_m-1} = q_{n_m-2}. \]
Similarly, if $b_{n_m}' = a_{n_m}$ and $Z_m<0$, then from line--4 we have $(-1)^{n_m-1} < 0$, hence $q_{n_m-1}^* = -q_{n_m-1}$ so that, by the inequality \eqref{b_n_m}, we have  
\[Z_m-b_{n_m}'q_{n_m-1}^* = -|Z_m| + b_{n_m}'q_{n_m-1} \le 0.\] 
Then the recursion relationship \eqref{q_k_recursion} and the inequality of line--3 will yield the inequality
\[|Z_m - b_{n_m}'q_{n_m-1}^*| + I_{<0}(Z_m- b_{n_m}'q_{n_m-1}^*) \le -(Z_m - b_{n_m}'q_{n_m-1}^*) + 1\]
\[=|Z_m| + b_{n_m}'q_{n_m-1}^* + 1 \le q_{n_m} - I_{<0}(Z_m) + b_{n_m}'q_{n_m-1}^* + 1\]
\[ = q_{n_m} - 1 - a_{n_m}q_{n_m-1} + 1 = q_{n_m} - a_{n_m}q_{n_m-1} = q_{n_m-2}.  \]
 In both cases, $b_{n_m}'$ would not satisfy the condition in line--7, hence we would have $b_{n_m} = b_{n_m}' = a_{n_m}$. Since $b_k=0$ whenever $k \notin \{n_m\}_0^M$, we conclude that for all $k$ we have $0 \le b_k \le a_k$.\\

\noindent To prove that $b_k=a_k$ implies that $b_{k+1}=0$, we let $k$ and $m$ are such that $n_m=k+1$. If $n_{m+1} \le k-1$ then $k \notin \{n_m\}_0^{M+1}$, hence $b_k=0 \le a_k-1$ so that we may assume that $n_{m+1} = n_m-1=k$. Again we will consider the two cases $n_m' \in \{n_m-1,n_m\}$ separately:\\ 

\noindent $\bullet$ When $n_m'=n_m-1$, we assume that $b_{k+1} \ge b_{k+1}' \ge 1$ and will prove that $b_k \le a_k-1$. We use the recursion formula \eqref{q_k_recursion}, the fact that the sequence $(q_k)$ is increasing and the inequality \eqref{Z_m+1_case1} to obtain
\[|Z_{m+1}| < q_{n_m-1} - q_{n_m-2} = q_k - q_{k-1} = (a_k-1)q_{k-1} + q_{k-2} < a_k{q_{k-1}}. \]
so when we assign $b_k'=b_{n_m-1}'=b_{n_{m+1}}'$ using the inequality \eqref{b_n_m}, we will have $b_k' \le a_k-1$. Furthermore, from formula \eqref{opposite_signs}, we obtain  
\[|Z_{m+1} - b_k'{q_{k-1}}^*| + I_{<0}(Z_{m+1} - b_k'{q_{k-1}}^*) \le |Z_{m+1}| - b_k'{q_{k-1}}\] 
\[\le |Z_{m+1}| - (a_k-1)q_{k-1} + 1 < q_k-q_{k-1} - (a_k-1)q_{k-1} + 1 = q_k - a_k{q_{k-1}}+1=q_{k-2}+1\]  
so that the condition of line--7 is not satisfied and $b_k := b_k'\le a_k-1$ as desired.\\

\noindent $\bullet$ When $n_m'=n_m$, we have
\[n_{m+2} \le n_m-2 = k-1 < n_{m+1} = n_m-1 = n_m'-1 = k < k+1 = n_m.\]
Suppose by contradiction that $b_k=a_k$ and $b_{k+1} \ge b_{k+1}' \ge 1$. Then by the recursion relationship \eqref{q_k_recursion}, the inequalities \eqref{b_n_m}, \eqref{Z_m+1_case2} and the assignment of line--15, we obtain the contradiction
\[q_k \le b_{k+1}'q_k = b_{n_m}'q_{n_m-1} \le |Z_m| = |b_{n_m}q_{n_m-1}^* + Z_{m+1}| = |b_{n_m}q_{n_m-1} - Z_{m+1}| \]
\[= |b_{n_m}q_{n_m-1} - b_{n_m-1}q_{n_m-2} + Z_{m+2}| \le b_{n_m}q_{n_m-1} - b_{n_m-1}q_{n_m-2} + |Z_{m+2}|\]
\[<  b_{n_m}q_{n_m-1}- b_{n_m-1}q_{n_m-2} +  (b_{n_{m+2}-1}+1)q_{n_{m+2}-2} = b_k{q_{k-1}} - b_{k-1}q_{k-2} + (b_{k-1}+1)q_{k-2} \]
\[= a_k{q_{k-1}} - b_{k-1}q_{k-2} + (b_{k-1}+1)q_{k-2} = q_k - (b_{k-1}+1)q_{k-2} + (b_{k-1}+1)q_{k-2} = q_k. \]
\end{proof}

\subsection{Examples}

When $\alpha$ is the golden section, we have $(q_k^*)_0^\infty := (1,-1,2,-3,5,...)$ and are able to extend Zeckendorf's Theorem to the integers. When $\alpha$ is the silver section, we have $(q_k^*)_0^\infty = (1,-2,5,-12,29,...)$. The following tables displays how the digits behave when counting from -24 to 24 using this base:\\

\noindent 
\small
\begin{tabular}{|c||c|c|c|}
\hline
 & $q_2^*$=$5$ & $q_1^*$=$-2$ & $q_0^*$=$1$\\ 
Z &    $b_3$ & $b_2$ & $b_1$\\
\hline\hline
1 & 0 & 0 & 1 \\
\hline
2  &0 & 0 & 2\\
\hline
3 & 1 & 1 & 0\\
\hline
4 & 1 & 1 & 1\\
\hline
5  & 1 & 0 & 0\\
\hline
6 & 1 & 0 & 1 \\
\hline
7 & 1 & 0 & 2\\
\hline
8 & 2 & 1 & 0\\
\hline
9 & 2 & 1 & 1\\
\hline
10 & 2 & 0 & 0\\
\hline
11  & 2 & 0 & 1\\
\hline
12 & 2 & 0 & 2\\
\hline
\end{tabular}\hspace{5pc}\begin{tabular}{|c||c|c|c|c|c|c|}
\hline
& $q_4^*$=$29$ & $q_3^*$=$-12$ & $q_2^*$=$5$ & $q_1^*$=$-2$ & $q_0^*$=$1$\\
Z & $b_5$ & $b_4$ & $b_3$ & $b_2$ & $b_1$\\
\hline\hline
13 &1& 1 & 0 & 2 & 0 \\
\hline
14 &1& 1 &0 & 2 & 1\\
\hline
15 &1&1 & 0 & 1 & 0\\
\hline
16 &1&1 & 0 & 1 & 1\\
\hline
17 &1& 1 & 0 & 0 & 0\\
\hline
18 &1&1 & 0 & 0 & 1 \\
\hline
19 &1&1 & 0& 0 & 2\\
\hline
20 &1&1 & 1 & 1 & 0\\
\hline
21 &1&1 & 1 & 1 & 1\\
\hline
22&1 &1 & 1 & 0 & 0\\
\hline
23&1 &  1 & 1 & 0 & 1\\
\hline
24&1 & 1 & 1 & 0 & 2\\
\hline
\end{tabular}\\

\vspace{2pc}

\begin{tabular}{|c||c|c|c|c|c|}
\hline
& $q_3^*$=$-12$ & $q_2^*$=$5$ & $q_1^*$=$-2$ & $q_0^*$=$1$\\ 
Z & $b_4$ & $b_3$ & $b_2$ & $b_1$\\
\hline\hline
-1 & 0 & 0 & 1 & 1\\
\hline
-2  & 0 & 0 & 1 & 0\\
\hline
-3  & 0 & 0 & 2 & 1\\
\hline
-4  & 0 & 0 & 2 & 0\\
\hline
-5  & 1 & 1 & 0 & 2\\
\hline
-6  & 1 & 1 & 0 & 1\\
\hline
-7  & 1 & 1 & 0 & 0\\
\hline
-8  & 1 & 1 & 1 & 1\\
\hline
-9  & 1 & 1 & 1 & 0\\
\hline
-10  & 1 & 0 & 0 & 2\\
\hline
-11  & 1 & 0 & 0 & 1\\
\hline
-12  & 1 & 0 & 0 & 0\\
\hline
\end{tabular}\hspace{3.6pc}\begin{tabular}{|c||c|c|c|c|}
\hline
& $q_3^*$=$-12$ & $q_2^*$=$5$ & $q_1^*$=$-2$ & $q_0^*$=$1$\\
Z&  $b_4$ & $b_3$ & $b_2$ & $b_1$\\
\hline\hline
-13  & 1 & 0 & 1 & 1\\
\hline
-14 & 1 & 0 & 1 & 0 \\
\hline
-15  & 1 &0 & 2 & 1\\
\hline
-16  &1 & 0 & 2 & 0\\
\hline
-17  &2 & 1 & 0 & 2\\
\hline
-18 & 2 & 1 & 0 & 1\\
\hline
-19 &2 & 1 & 0 & 0 \\
\hline
-20 &2 & 1 & 1 & 1\\
\hline
-21 &2 & 1 & 0 & 2\\
\hline
-22 &2 & 0 & 0 & 2\\
\hline
-23 &2 & 0 & 0 & 1\\
\hline
-24 &  2 & 0 & 0 & 0\\
\hline
\end{tabular}
\normalsize

\section{Appendix -- Mathematica Implementation} 

We use Mathematica$^{\operatorname{TM}}$ to implement the algorithm \ref{abs} and \ref{alt} with the base whose first continued fraction partial quotients are $(a_k := k)_1^9$. The vectors $\bf{b}$ and $\bf{c}$ start at position 1 and the vectors $\bf{q}$ and  $\bf{q}^*$ start in positions $1$ so that we obtain the dot product representation 
\[N = \bf{c}\cdot\bf{q} = \operatorname{Ost}(N)\cdot\bf{q}\]
and  
\[Z = \bf{b}\cdot\bf{q}^* = \operatorname{AltOst}(Z)\cdot\bf{q}^*.\]  

\noindent\includegraphics[scale=.8]{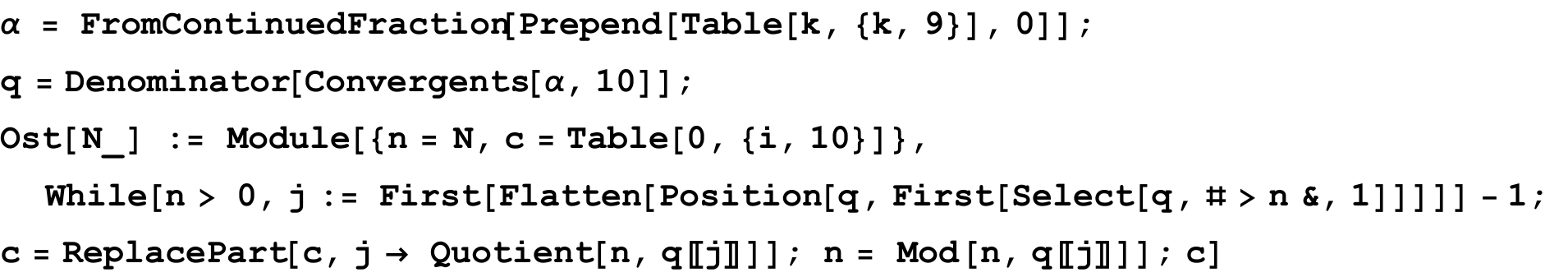}\\

\noindent\includegraphics[scale=.8]{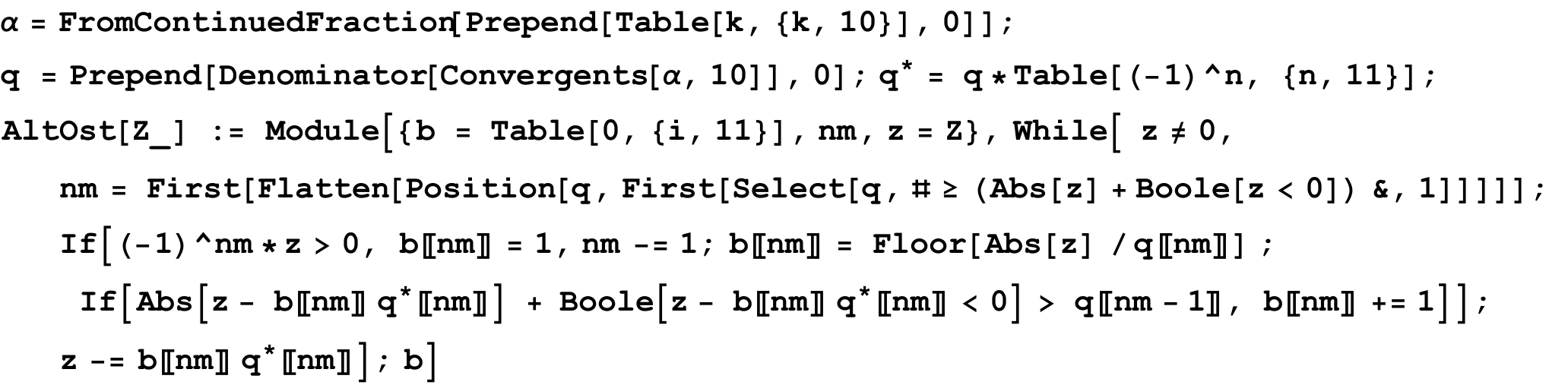}\\

\section{Acknowledgments}

\noindent This work could have not been completed without the guidance, encouragement and good company of Robbie Robinson from George Washington University.

\end{document}